\documentclass[11pt]{article}
\usepackage[top=2cm,bottom=2.5cm,right=2.5cm,left=2.5cm]{geometry}
\usepackage[english]{babel}
\usepackage[utf8]{inputenc}
\usepackage{times}
\usepackage{mathtools, bm}
\usepackage{amssymb, bm}
\usepackage{graphicx}
\usepackage{mathrsfs}
\usepackage{stmaryrd}
\usepackage{amsthm}
\usepackage{listings}
\usepackage{pict2e}
\usepackage{pgf, tikz}
\usepackage{dsfont}
\usepackage{algorithm2e}
\usepackage{algorithmic}
\usepackage{multicol}
\usepackage{hyperref}

\allowdisplaybreaks

\thicklines

{
      \theoremstyle{plain}
      \newtheorem{assumption}{Assumption}
  }

\newtheorem{theorem}{Theorem}

\newtheorem{lemma}{Lemma}
\newtheorem{remark}{Remark}
\newtheorem{definition}{Definition}
\newtheorem{example}{Example}

\numberwithin{equation}{section} 
\numberwithin{lemma}{section} 
\numberwithin{remark}{section} 
\numberwithin{example}{section}
\numberwithin{corollary}{section}
\numberwithin{proposition}{section}

\date{}
\author{Pierre Le Bris\footnote{Sorbonne Université, CNRS UMR 7598, Laboratoire Jacques-Louis Lions, F-75005 Paris, France. Email : pierre.lebris[AT]sorbonne-universite.fr}$\ $ and Christophe Poquet\footnote{Université Claude Bernard Lyon 1, CNRS UMR 5208, Institut Camille Jordan, F-69622 Villeurbanne, France. Email : poquet[AT]math.univ-lyon1.fr}}
\title{A note on uniform in time mean-field limit in graphs}

\begin{document}

\maketitle

\begin{abstract}
In this article we show, in a concise manner, a result of uniform in time propagation of chaos for non exchangeable systems of particles interacting according to a random graph. Provided the interaction is Lipschitz continuous, the restoring force satisfies a general one-sided Lipschitz condition (thus allowing for non-convex confining potential) and the graph is dense enough, we use a coupling method suggested by Eberle \cite{Eberle16} known as \textit{reflection} coupling to obtain uniform in time mean-field limit with bounds that depend explicitly on the graph structure.
\end{abstract}

%
%
%
%

\section{Introduction}

%
%
%
%

\subsection{Model and motivation}

Let $N\in\mathbb{N}$ and consider an adjacency matrix $\xi^{(N)}=\left(\xi^{(N)}_{i,j}\right)_{i,j\in\left\{1,...,N\right\}}$ with coefficients $\xi^{(N)}_{i,j}\in\left\{0,1\right\}$. Denote by $G^{(N)}=(V^{(N)},E^{(N)})$ the graph associated to this adjacency matrix, in the sense ${V^{(N)}:=\{1,...,N\}}$ and $E^{(N)}:=\left\{(i,j)\in V^{(N)}\times V^{(N)}\text{ s.t.  }\xi^{(N)}_{i,j}=1\right\}$. We assume by convention that $\xi^{(N)}_{i,i}=0$ for all $i\in\left\{1,...,N\right\}$.

We will consider in this note a system of particles interacting according to this graph, more precisely the system of $N$ SDEs in $\mathbb{R}^d$
\begin{equation}\label{eq:IPS}
\tag{IPS}
    dX^i_t=F\left(X^i_t, \omega_i\right)dt+\frac{\alpha_N}{N}\sum_{j=1}^N\xi^{(N)}_{i,j}\Gamma\left(X^i_t,\omega_i,X^j_t,\omega_j\right)dt+\sqrt{2}\sigma dB^i_t,\ \ \ i\in\{1,...,N\},
\end{equation}
where $\left(B^i_\cdot\right)_{i=1,...,N}$ is a sequence of independent standard Brownian motions, $\{\omega_i\}_{i\in\{1,...,N\}}$ is a sequence of elements in $\mathbb{R}^{d'}$ (with the convention $d'=0$ if $\Gamma$ does not depend on $\omega$) which represents some environmental disorder, $(\alpha_N)_{N\geq 1}$ is a positive scaling, $F:\mathbb{R}^d\times \mathbb{R}^{d'}\mapsto\mathbb{R}^d$ is an outside force, $\Gamma:\left(\mathbb{R}^{d}\times \mathbb{R}^{d'}\right)^2\mapsto\mathbb{R}^d$ is an interaction kernel and $\sigma$ is a positive diffusion coefficient. 
We will assume that $(\omega_i)_{i=1,\ldots,N}$ is a sequence of IID random variables, and that the Brownian motions are independent from the initial condition $(X^i_0,\omega_i)_{i=1\ldots, N}$. We will denote by $\mathbb{E}$ the expectation with respect to the Brownian motions, the initial condition and the disorder.

One of the main difficulties arising in the study of this model comes from the fact that the particles are not exchangeables as, \textit{a priori}, some may interact with more particles than others. This motivates us to consider the empirical distribution, defined for $(X^1_t,...,X^N_t)$ a solution of \eqref{eq:IPS} with disorder $(\omega_1,\ldots,\omega_N)$, by
\begin{align*}
    \mu^N_t:=\frac{1}{N}\sum_{i=1}^N\delta_{(X^i_t,\omega_i)}.
\end{align*}
Notice that $\mu^N_t$ is a random variable.

We are interested in the limit, as the number $N$ of particles goes to infinity, of \eqref{eq:IPS}. Intuitively one expects the empirical measure to converge towards a measure $\bar{\rho}$ which would represent the law of one typical particle and its disorder within a cloud of interacting disordered particles. Assuming that $\frac{\alpha_N}{N}\sum_{j=1}^N\xi^{(N)}_{i,j}$ converges in some sense (given later) to a parameter $p$, this typical particle $\bar X^\omega$ with disorder $\omega$ would then in the limit evolve according to the non-linear diffusion
\begin{equation}\label{eq:NL}
\tag{NL}
    \left\{
    \begin{array}{l}
        d\bar{X}^\omega_t=F\left(\bar{X}^\omega_t,\omega\right)dt+p\int_{\mathbb{R}^d\times\mathbb{R}^{d'}}\Gamma\left(\bar{X}^\omega_t,\omega, y,\tilde{\omega}\right)\bar{\rho}_t(dy,d\tilde{\omega})dt+\sqrt{2}\sigma dB_t,   \\
        \bar{\rho}_t=\text{Law}(\bar{X}^\omega_t,\omega)
    \end{array}
    \right.,
\end{equation}
where $B$ is a standard Brownian motion. This limit was proven rigorously on finite time horizon $[0,T]$, where $T$ does not depend on $N$, under some hypotheses on the graph structure, which are in particular satisfied by sufficiently dense Erd\H{o}s-Rényi graphs \cite{DGL16,CDG20}. Our aim in the present paper is to obtain uniform in time estimates of the distance between the empirical measure $\mu^N_t$ and the limit distribution $\bar \rho_t$, with estimates that depend explicitly on the graph. 

Note that proofs of convergence of particle systems interacting via random graphs possessing a spatial structure (for example converging to a graphon) were recently obtained \cite{OR19,luccon2020quenched,bayraktar2020graphon,BW21}. 
In particular uniform in time estimates in the context of graphons were obtained in \cite{BW21}, where the empirical measure is shown to be close to the limit distribution with high probability with respect to the distribution of the random graph. In this note we aim at obtaining {\sl quenched} results, i.e. obtaining estimates that hold for almost every realization of the graph.

These recent results generalize the classical case of complete graph of interaction ($\alpha_N=1$ and ${\xi\equiv1}$) and without any dependence on the environment $\omega$, for which it is well known that under some weak conditions on $F$ and $\Gamma$ the empirical measure $\mu^N_t$ converges towards the non-linear limit $\bar{\rho}_t$ \cite{Meleard96, Sznitman91}. This phenomenon has been named \textit{propagation of chaos}, an idea motivated by Kac \cite{Kac56} : it is equivalent, in the case of exchangeable particles, to the convergence of all $k$ marginals of the law of $(X^1_t,...,X^N_t)$ to $\bar{\rho}_t^{\otimes k}$ (the non linear limit tensorized $k$ times). Thus, as $N$ goes to infinity, two particles become "more and more" independent, converging to a tensorized law, hence \textit{chaos}. The term \textit{propagation} emphasizes the fact that it is sufficient to show independence at the limit at time 0 for it to also hold true at the limit at later time $t>0$. We refer to the recent \cite{CD22-1, CD22-2}, and references therein, for a thorough reviews on propagation of chaos.

To quantify the convergence of the empirical measure towards the non-linear limit, we will use the $L^1$-Wasserstein distance defined as follows.


\begin{definition}
For $\mu$ and $\nu$  two probability measures on $\mathbb{R}^{d}$, denote by $\Pi\left(\mu,\nu\right)$ the set of couplings of $\mu$ and $\nu$, i.e. the 
set of probability measures $\pi$ on $\mathbb{R}^{d}\times\mathbb{R}^{d}$ with $\pi(A\times \mathbb{R}^{d}) = \mu(A)$ and $\pi(\mathbb{R}^{d}\times 
A ) = \nu(A)$ for all Borel set $A$ of $\mathbb{R}^{d}$. The $L^1$-Wasserstein distance is given by   
\begin{equation}\label{eq:def_wass}
\mathcal{W}_1\left(\mu,\nu\right)=\inf_{\pi\in\Pi\left(\mu,\nu\right)}\int|x-\tilde{x}|\pi\left(dxd \tilde{x} \right).
\end{equation}
\end{definition}

Equivalently, we may write in probabilistic terms
\begin{equation*}
\mathcal{W}_1\left(\mu,\nu\right)=\inf_{X\sim\mu,Y\sim\nu}\mathbb{E}\left(\left|X-Y\right|\right),
\end{equation*}
where $X\sim\mu$ is a random variable distributed according to $\mu$. This distance is a usual distance in optimal transport and in the study of measures in general, as the space of probability measures on $\mathbb{R}^d$, equipped with the $L^1$-Wasserstein distance, is a complete and separable metric space (see for instance \cite{Bolley08}). To prove the convergence in Wasserstein distance, we use a \textit{coupling method}. The idea is, instead of considering the minimum over all couplings of the law of the particle system and the non-linear limit as should be done according to \eqref{eq:def_wass}, we construct simultaneously two solutions of \eqref{eq:IPS} and \eqref{eq:NL} such that the expectation of the $L^1$ distance between these solutions tends to decrease. We would thus construct a specific coupling, that controls the $L^1$-Wasserstein distance, providing a quantitative bound. To construct this coupling, we may act on the Brownian motions and on the random variables $\omega$.

The approach we consider was motivated by the work of Eberle \cite{Eberle16}. Let us describe the idea of the coupling method. Assume, for the sake of the explanation, that $F=-\nabla U$ where $U$ is therefore a confinement potential. Constructing a solution of \eqref{eq:IPS} and $N$ independent solutions of \eqref{eq:NL} simultaneously by choosing the same Brownian motions yields the so-called \textit{synchronous} coupling, for which the Brownian noise cancels out in the infinitesimal evolution of the  difference $Z^i_t=X^i_t-\bar{X}^{\omega,i}_t$. In that case the contraction of a distance between the processes can only be induced by the deterministic drift. Such a deterministic contraction only holds under very restrictive conditions, in particular $U$ should be strongly convex. In the case of a non-convex confinement potential $U$, it is necessary to make use of the noise to obtain contraction. Constructing the solutions choosing the two Brownian motions to be antithetic (or opposite) in the direction of space given by the difference of the processes maximises the variance of the noise in the desired direction. However, \textit{a priori}, nothing ensures the noise will bring the processes closer rather than further. We thus modify the Euclidean distance by some concave function $f$, in order for a random decrease of the difference to have more effect than a random increase of the same amount.

This method was originally designed to deal with the long time behavior of general diffusion processes, as in \cite{Eberle16,EGZ19}, and later extended to show uniform in time propagation of chaos in a mean-field system in \cite{DEGZ20}. The main difference of this work when compared to \cite{DEGZ20} comes from the non-exchangeability of the particles, as we thus need careful estimates with respect to the graph. For instance, since the particles do not share a common law, we cannot restrict our analysis to the study of $\mathbb{E}[Z^1_t]$ and then conclude using the fact that all $Z^i_t$ have the same expectation ; the proof requires a more global approach to the system, by considering the empirical measure, and thus other tools.

The framework of this article was inspired by \cite{DGL16}, and we improve their result, obtaining a uniform in time estimate, while removing some of the boundedness assumptions on the various functions.

Uniform in time propagation of chaos has recently attracted a lot of attention. The ideas behind this coupling method were used to prove such estimates in a kinetic setting (i.e a particle is represented by both its position and velocity, and the Brownian motion only acts on the latter) in \cite{GLBM22, Schuh22}. In \cite{Mal01,Mal03}, uniform in time propagation of chaos was proved using synchronous coupling assuming a convexity condition on the interaction.  Likewise, a similar result was obtained in \cite{CGM08}, using functional inequalities, under some assumptions of convexity at infinity. Also using functional inequalities for mean field models developed in \cite{GLWZ22}, uniform in time propagation of chaos was proved in a kinetic setting in \cite{GLWZ21, GM21, Mon17} combining the hypocoercivity approach with uniform in the number of particles logarithmic Sobolev inequalities. Let us also mention the optimal coupling approach of \cite{Sal20} using a WJ inequality, which is also used in \cite{DT19}, which enables to recover the results in \cite{DEGZ20}. Thanks to an analysis of the relative entropy through the BBGKY hierarchy, building upon the work \cite{Lac21}, a result of uniform in time propagation of chaos was obtained, with a sharp rate in $N$, in \cite{LLF22}. Finally, in the recent work \cite{DT21},  uniform in time weak propagation of chaos (i.e observable by observable) was shown on the torus via Lions derivative. Notably, this result may extend to the case the McKean-Vlasov limit has several invariant measures, as in the Kuramoto model for instance.

All the works mentioned above assume the interaction to be "sufficiently nice" (either gradient of a convex potential, smooth, bounded, etc) and not according to a random graph. The case of singular interactions is, because of the various applications in biology, physics and others, also of great interest. Though some recent works have obtained quantitative mean-field convergence for some singular potential (for instance using entropy dissipation in \cite{JW18}, modulated energy in \cite{Ser20}, a mix of both in \cite{BJW19}, or BBGKY hierarchies in \cite{BJS22}), few still have obtained uniform in time estimates. We mention the results dealing with singular repulsive interactions of the type $-\log|x|$ or $|x|^{-s}$, $0<s<d-2$, in \cite{RS21} using the modulated energy, dealing with the specific case of the 2D vortex model in \cite{GLBM21-2} (building upon the work \cite{JW18}), or dealing with repulsive singular interactions in dimension one in \cite{GLBM22-2} using another type of coupling method.

There again, the particles are not interacting according to a graph.

%
%
%
%

\subsection{Assumptions and main result}

Denote by $d_i^{(N)}:=\sum_{j=1}^N\xi^{(N)}_{i,j}$ and $\tilde d_i^{(N)}:= \sum_{j=1}^N \xi^{(N)}_{j,i}$ the degrees of vertex $i$. The family $\xi^{(N)}$ may be deterministic or random, in this second case we assume that it is independent from the Brownian motions and from $(X^i_0,\omega_i)_{i=1,\ldots,N}$ and that the following assumption is verified almost surely. These assumptions are similar to the ones made in \cite{DGL16}.

\begin{assumption}[On the graph]\label{hyp:graph}
The adjacency matrix $\xi^{(N)}$ satisfies the following assertions for all $N\geq1$.
\begin{description}
\item[\ref{hyp:graph}-1] There exists a positive constant $C_{g}$ such that
\begin{align*}
   \limsup_{N \rightarrow \infty}\ D_{N,g} \leq C_g,
\end{align*}
where
\begin{align*}
   D_{N,g}:=\sup_{i\in \{1,\ldots,N\}} \alpha_N\left(\frac{d^{(N)}_i}{N}+\frac{\tilde d^{(N)}_i}{N}\right).
\end{align*}
\item[\ref{hyp:graph}-2]There exists $p\in[0,1]$ such that
\begin{align*}
I_{N,g}:=\sup_{i\in\{1,...,N\}}\left|\alpha_N\frac{d^{(N)}_i}{N}-p\right|\xrightarrow[N\rightarrow\infty]{a.s}0.
\end{align*}
\end{description}
\end{assumption}


\begin{example}~
\begin{itemize}
    \item Regular graphs: if $\xi^{(N)}$ defines a regular graph of degree $d_N$ with $\frac{d_N}{N} \xrightarrow[N\rightarrow\infty]{} p$, then $\xi^{(N)}$ satisfies Assumption~\ref{hyp:graph} with $\alpha_N=1$.
    \item Erd\H{o}s-Rényi graphs: Let $\xi^{(N)}_{i,j}$ be a sequence of IID Bernouilli variables of parameter $q_N$ with either $q_N \xrightarrow[N\rightarrow\infty]{} p$ or $\frac{1}{q_N} = o\left(\frac{N}{\log N}\right)$. Then $\xi^{(N)}$ satisfies Assumption~\ref{hyp:graph} with $\alpha_N=1$ in the first case, $\alpha_n=\frac{1}{p_N}$ and $p=1$ in the second case.
    \item Community models: more generally, suppose that the whole population is divided in $r$ sub-populations of size $m$ (so that $N=rm$), the graph structure being then defined by independent random variables $\xi^{(N,k,k')}_{i,j}$ for $k,k'\in \{1,\ldots, r\}$ and $i,j\in \{1,\ldots,m\}$. Suppose moreover that the intra-community interaction variables $\xi^{(N,k,k)}_{i,j}$ are of Bernoulli distribution with parameter $q_N$ satisfying $\frac{1}{q_N} = o\left(\frac{N}{\log N}\right)$, while the inter-community interaction variables $\xi^{(N,k,k')}_{i,j}$ are of Bernoulli distribution with parameter $q^{k,k'}_N$ satisfying $q^{k,k'}_N=o(q_N)$. Then, for $r$ fixed and $m\rightarrow \infty$, $\xi^{(N)}$ satisfies Assumption~\ref{hyp:graph} with $\alpha_N=\frac{1}{q_N}$ and $p=\frac{1}{r}$. For more details and the proof of this result see Appendix~\ref{sec:app}.
\end{itemize}
\end{example}


\begin{assumption}[On the restoring force]\label{hyp:F}
There exists a continuous function $\kappa:\mathbb{R}^+\mapsto \mathbb{R}$ satisfying \linebreak $\liminf_{r\rightarrow\infty}\kappa(r)>0$ such that
\begin{equation*}
    \forall x,y\in\mathbb{R}^d,\ \ \ \forall \omega\in\mathbb{R}^{d'},\ \ \ (F(x,\omega)-F(y,\omega))\cdot(x-y)\leq- \kappa(|x-y|)|x-y|^2.
\end{equation*}
In particular, this implies that there exist $M_F\geq0$ and $m_F>0$ such that
\begin{align*}
    \forall x,y\in\mathbb{R}^d,\ \ \ \forall \omega\in\mathbb{R}^{d'},\ \ \ (F(x,\omega)-F(y,\omega))\cdot(x-y)\leq M_F- m_F|x-y|^2.
\end{align*}
\end{assumption}

The added one-sided assumption on $F$ when compared to \cite{DGL16} is both classical (see \cite{DEGZ20}) and necessary to ensure that the particles tend to come back to a compact set.


\begin{example}
Let us give some examples of functions $F$ satisfying Assumption~\ref{hyp:F}. Let $F(x,\omega)=-V'(x)$ in dimension 1 with :
\begin{itemize}
    \item  $V(x)=\frac{x^2}{2}$ : then $F$ satisfies Assumption~\ref{hyp:F} with $\kappa\equiv 1$. 
    \item $V(x)=\frac{x^4}{4}-\frac{x^2}{2}$ : then
    \begin{align*}
        (F(x)-F(y))(x-y)=&-(x^3-y^3)(x-y)+(x-y)^2\\
        =&-(x-y)^2\left(x^2+xy+y^2-1\right)\\
        \leq& -(x-y)^2\left(\frac{1}{4}(x-y)^2-1\right).
    \end{align*}
    Hence,  $F$ satisfies Assumption~\ref{hyp:F} with $\kappa(x)=\frac{x^2}{4}-1$. 
\end{itemize}
Likewise, we may consider disordered restoring forces such as ${F(x,\omega)=-x^3+\omega x}$, provided $\omega$ belongs to a bounded subset of $\mathbb{R}$, or ${F(x,\omega)=-\omega x^3}$ provided $\omega$ is positive bounded from below. 
\end{example}


\begin{assumption}[On the interaction]\label{hyp:gamma_lip}
$\Gamma$ satisfies \ref{hyp:gamma_lip}-1 below, and either \ref{hyp:gamma_lip}-2 or \ref{hyp:gamma_lip}-2-bis.
\begin{description}
\item[\ref{hyp:gamma_lip}-1] $\Gamma : (x,\omega, y, \omega')\rightarrow\Gamma(x,\omega, y, \omega')$ is Lipschitz-continuous in $(x,y)$ uniformly in $\omega$ and $\omega'$: 
\begin{align*}
    \exists L_{\Gamma}\geq0,\ \forall x,y,t,s \in\mathbb{R}^d,\ \forall \omega,\omega'\in\mathbb{R}^{d'},&\\ \left|\Gamma(x,\omega,t,\omega')-\Gamma(y,\omega,s,\omega')\right|&\leq L_\Gamma \left(f(|x-y|)+f(|t-s|)\right),
\end{align*}
where $f$ is a function given below in \eqref{eq:def_f} such that $x\mapsto f(|x|)$ is equivalent to the usual $L^1$ distance in $\mathbb{R}$. 

Furthermore, for simplicity, we have $\Gamma(0,0,0,0)=0$.

\item[\ref{hyp:gamma_lip}-2] $\Gamma$ is Lipschitz-continuous in $\omega$ and $\omega'$ at $(x,y)=(0,0)$:
\begin{align*}
    \exists L_{\Gamma}\geq0,\ \forall \omega_1,\omega'_1,\omega_2,\omega'_2\in\mathbb{R}^{d'},&\\ \left|\Gamma(0,\omega_1,0,\omega'_1)-\Gamma(0,\omega_2,0,\omega'_2)\right|&\leq L_\Gamma \left(|\omega_1-\omega_2|+|\omega'_1-\omega'_2|\right).
\end{align*}

\item[\ref{hyp:gamma_lip}-2-bis] $\Gamma$ is bounded
\begin{align*}
    \exists L_{\infty}\geq0,\ \forall x,y \in\mathbb{R}^d,\ \forall \omega,\omega'\in\mathbb{R}^{d'},\ |\Gamma(x,\omega,y,\omega')|\leq L_{\infty}.
\end{align*}
\end{description}
\end{assumption}

These are usual assumptions when proving mean-field limits using coupling methods. In particular, Assumptions~\ref{hyp:F}~and~\ref{hyp:gamma_lip} imply strong existence and uniqueness for the solutions of both \eqref{eq:IPS} and \eqref{eq:NL}.


\begin{assumption}[On the initial distributions]\label{hyp:rho_0}~
\begin{description}
\item[\ref{hyp:rho_0}-1] The sequence of disorder $(w_i)_{i=1,\ldots,N}$ is IID of distribution $\nu$, satisfying
\begin{align*}
     \int_{\mathbb{R}^{d'}} \left(|\omega|^2+|F(0,\omega)|^2\right) \nu(d\omega)\leq  C_{\mathrm{dis}}.
\end{align*}
\item[\ref{hyp:rho_0}-2] The random variables $(X^i_0)_{i=1,\ldots,N}$ are exchangeable, independent from the disorder $(w_i)_{i=1,\ldots,N}$ and satisfy
\begin{align*}
    \mathbb{E}\left( |X^1_0|\right) <\infty.
\end{align*}
\item[\ref{hyp:rho_0}-3] The initial distribution $\bar \rho_0$ is a product measure with second marginal equal to $\nu$, i.e. $\bar{\rho}_0(dx,d\omega)=\bar{\rho}^1_0(dx)\nu(d\omega)$.
Moreover there exists a positive constant $\bar C$ such that
\begin{align*}
     \int_{\mathbb{R}^d} |x|^2 \bar{\rho}^1_0(dx)\leq \bar C.
\end{align*}
\end{description}
\end{assumption}

We may now state the main theorem.


\begin{theorem}\label{thm:main}
Consider Assumptions~\ref{hyp:graph},\ref{hyp:F},~\ref{hyp:gamma_lip}~and~\ref{hyp:rho_0}. There exist explicit positive constants $c_\Gamma,\tilde{C},\tilde{c}$ that do not depend on $N$ and the graph such that for all $t\geq0$ and all $N\geq 1$, provided $L_\Gamma\leq c_\Gamma/D_{N,g}$ (recall that $\limsup_{N\rightarrow \infty}D_{N,g}<C_g$),  
\begin{equation}\label{eq:main_thm}
    \mathbb{E}\mathcal{W}_1\left(\mu^N_t,\bar{\rho}_t\right)\leq \tilde{C}\left(e^{-\tilde{c}t}\mathbb{E}\mathcal{W}_1\left(\mu^N_0,\bar{\rho}_0\right)+L_\Gamma\sqrt{\frac{\alpha_N D_{N,g}}{N}}+L_\Gamma I_{N,g}+h(N)\right),
\end{equation}
where $h:\mathbb{N}\mapsto\mathbb{R}^+$ is an explicit decreasing function such that $h(N)\xrightarrow[N\rightarrow\infty]{}0$ that only depends on the dimensions $d$ and $d'$ and the second moment of $\rho$ and $\bar{\rho}_0$.
\end{theorem}


\begin{remark}
The smallness assumption on the Lipschitz coefficient of $\Gamma$ is natural to obtain uniform in time propagation of chaos, as for large interactions the non linear limit may have several stationary measures (see \cite{HT10} for instance). Non uniqueness of the stationary measures of \eqref{eq:NL} prevents time-uniform estimate for the mean field limit, since on the other hand there is uniqueness of the stationary distribution of \eqref{eq:IPS}.
\end{remark}


\begin{remark}
We may write the order of magnitude of the rate function $h$ depending on the dimension
\begin{align*}
    h(N)\lesssim& N^{-\frac{1}{3}}\mathds{1}_{d+d'\leq2}+N^{-\frac{1}{d+d'}}\mathds{1}_{d+d'\geq3}.
\end{align*}
In reality, this term is a consequence of the approximation of the measure $\bar{\rho}_t$ by the empirical measure given by $N$ independent random variables distributed according to $\bar{\rho}_t$, as it is given by \cite{FG15}. We notice that it could be improved, however at a cost, as there is a tradeoff between the speed of convergence and the moments we impose on the initial condition. If we assumed that $\bar{\rho}_0$ admits a $q$-th moment with $q>2$, we could show the following bound:
\begin{align*}
    h(N)\lesssim& N^{-\frac{1}{2}}\mathds{1}_{d+d'=1}+N^{-\frac{1}{2}}\log(1+N)\mathds{1}_{d+d'=2}+N^{-\frac{1}{d+d'}}\mathds{1}_{d+d'\geq3}.
\end{align*}
\end{remark}


\begin{remark}
If Assumption~\ref{hyp:gamma_lip}-2-bis holds instead of Assumption~\ref{hyp:gamma_lip}-2, the coefficients $L_\Gamma$ within the parentheses in \eqref{eq:main_thm} are to be replaced by $L_\infty$.
\end{remark}


\begin{remark}
Although we do not write the calculations for the sake of conciseness, a similar theorem can be proved if $p=0$ for weaker assumptions on $\Gamma$. In this case, the limit is a linear Ornstein-Uhlenbeck process, and we do not rely on a form of Law of Large Number to get the limit as $N$ goes to infinity, we only have to show that the interaction term vanishes sufficiently fast. Thus we need the expectation of $\Gamma$ to be bounded uniformly in time, which can be done under weaker assumptions and in particular doesn't require a Lipschitz assumption, as the convergence to 0 of $\frac{\alpha_N}{N}\sum_{j=1}^N\xi^{(N)}_{i,j}$ then yields the result of propagation of chaos.
\end{remark}

%
%
%
%

\subsection{Semimetric and preliminary results}

As mentioned previously, we use a concave function to modify the Euclidean distance in order to use the reflection coupling. Define
\begin{align*}
    R_0:=&\inf\left\{s\geq0\ :\ \forall r\geq s,\ \kappa(r)\geq0 \right\},\\
    R_1:=&\inf\left\{s\geq R_0\ :\ \forall r\geq s,\ s(s-R_0)\kappa(r)\geq8\sigma^2 \right\},
\end{align*}
and the functions
\begin{align*}
\phi(r):=&\exp\left(-\frac{1}{4\sigma^2}\int_0^rs\kappa_-(s)ds\right),\\
\Phi(r):=&\int_0^r\phi(s)ds,\\
g(r):=&1-\frac{c}{2}\int_0^{r\wedge R_1}\Phi(s)\phi(s)^{-1}ds,
\end{align*}
where $\kappa_-=\max(0,-\kappa)$ and $c=\left(\int_0^{R_1}\Phi(s)\phi(s)^{-1}ds\right)^{-1}.$
Finally, define 
\begin{equation}\label{eq:def_f}
    f(x)=\int_0^xg(t)\phi(t)dt.
\end{equation}
Note that $\phi$ and $g$ are positive non-increasing on $\mathbb{R}^+$ and that $\phi(r)=\phi(R_0)\leq 1$ for $r\geq R_0$, and $g(r)=\frac{1}{2}$ for $r\geq R_1$. In particular, for $r\geq R_1$ we simply have $f(r)=f(R_1)+\frac{\phi(R_0)(r-R_1)}{2}$. The function $f$ satisfies moreover some useful properties gathered in the following Lemma, the proof of which can be found in \cite{DEGZ20}.


\begin{lemma}[Some properties of the semimetric]\label{lem:semi-metric}
The function $f$ satisfies the following properties :
\begin{itemize}
    \item $f:\mathbb{R}^+\rightarrow \mathbb{R}^+$ is non-negative and increasing. Furthermore $0< f'(x)\leq 1$ for all $x\geq0$.
    \item There exist $c_f, C_f>0$ such that for all $x\in\mathbb{R}$, we have $c_f|x|\leq f(|x|)\leq C_f|x|$.
    \item We have
    \begin{equation}\label{ineq derivatives f}
        \forall r\in\mathbb{R}^+\setminus \{R_1\},\ \ \   f''(r)-\frac{1}{4\sigma^2}r\kappa(r)f'(r)\leq-\frac{c}{2}f(r).
    \end{equation}
\end{itemize}
\end{lemma}


We now give a uniform in time moment bound for the non linear process \eqref{eq:NL}, relying in particular on Assumption~\ref{hyp:F}. 

\begin{lemma}[Uniform in time bound on the second moment]\label{lem:mom_nl}
Consider Assumption~\ref{hyp:F},~\ref{hyp:gamma_lip}~and~\ref{hyp:rho_0}, and let $(\bar{X}^{\omega}_t)_t$ be the unique strong solution of \eqref{eq:NL}. Assuming $2pC_fL_\Gamma<m_F$, there exists a constant $\bar{\mathcal{C}}_2>0$ such that for all $t\geq0$
\begin{align*}
    \mathbb{E}\left(|\bar{X}^\omega_t|^2\right)\leq\bar{\mathcal{C}}_2.
\end{align*}
\end{lemma}


\begin{proof}
Using Itô's formula on the function $H(x)=\frac{x^2}{2}$, we obtain
\begin{equation}\label{eq:Ito_H}
    dH\left(\bar{X}^{\omega}_t\right)=A_tdt+dM_t,
\end{equation}
where $(M_t)_t$ is a continuous local martingale and 
\begin{align*}
    A_t=\bar{X}^{\omega}_t\cdot F\left(\bar{X}^{\omega}_t,\omega\right)+p\int_{\mathbb{R}^d\times\ \mathbb{R}^{d'}}\bar{X}^{\omega}_t\cdot\Gamma\left(\bar{X}^{\omega}_t,\omega,y,\bar{\omega}\right)\bar{\rho}_t(dy,d\bar{\omega})+\sigma^2 d.
\end{align*}
First, using Assumption~\ref{hyp:F},
\begin{align*}
    \bar{X}^{\omega}_t\cdot F\left(\bar{X}^{\omega}_t,\omega\right)\leq M_F-m_F\left|\bar{X}^{\omega}_t\right|^2+\bar{X}^{\omega}_t\cdot F\left(0,\omega\right).
\end{align*}
Then, using Assumption~\ref{hyp:gamma_lip}
\begin{align*}
    \bar{X}^{\omega}_t\Gamma\left(\bar{X}^{\omega}_t,\omega,y,\bar{\omega}\right)=&\bar{X}^{\omega}_t\cdot\left(\Gamma\left(\bar{X}^{\omega}_t,\omega,y,\bar{\omega}\right)-\Gamma\left(0,\omega,0,\bar{\omega}\right)\right)+\bar{X}^{\omega}_t\cdot\left(\Gamma\left(0,\omega,0,\bar{\omega}\right)-\Gamma\left(0,0,0,0\right)\right)\\
    \leq& C_fL_\Gamma\left|\bar{X}^{\omega}_t\right|\left(\left|\bar{X}^{\omega}_t\right|+\left|y\right|\right)+L_\Gamma\left|\bar{X}^{\omega}_t\right|\left(\left|\omega\right|+\left|\bar{\omega}\right|\right).
\end{align*}
Note that if Assumption~\ref{hyp:gamma_lip}-2-bis holds rather than Assumption~\ref{hyp:gamma_lip}-2, the term above can directly be bounded by $L_\infty |\bar{X}^{\omega}_t|$. Then
\begin{align*}
    p\int_{\mathbb{R}^d\times\mathbb{R}^{d'}}\bar{X}^{\omega}_t\cdot&\Gamma\left(\bar{X}^{\omega}_t,\omega,y,\bar{\omega}\right)\bar{\rho}_t(dy,d\bar{\omega})\\
    \leq& pC_fL_\Gamma\left|\bar{X}^{\omega}_t\right|^2+pC_fL_\Gamma\left|\bar{X}^{\omega}_t\right|\mathbb{E}\left|\bar{X}^{\omega}_t\right|+pL_\Gamma\left|\bar{X}^{\omega}_t\right|\left|\omega\right|+pL_\Gamma C_{dis}^{1/2}\left|\bar{X}^{\omega}_t\right|,
\end{align*}
where for this last term we used Assumption~\ref{hyp:rho_0}. Finally
\begin{align*}
    \mathbb{E}A_t\leq M_F+\sigma^2 d- \left(m_F-2pC_fL_\Gamma\right)\mathbb{E}\left(\left|\bar{X}^{\omega}_t\right|^2\right)+\mathbb{E}\left(\left|\bar{X}^{\omega}_t\right|\left(\left|F\left(0,\omega\right)\right|+pL_\Gamma|\omega|+pL_\Gamma C_{dis}^{1/2}\right)\right).
\end{align*}
Assuming $2pC_fL_\Gamma<m_F$, using the inequality $\forall x, y\in\mathbb{R},\ \forall \alpha>0,\ xy\leq\frac{\alpha x^2}{2}+\frac{y^2}{2\alpha}$, we can ensure there exist two non negative constant $B_1$ and $B_2$ such that
\begin{align*}
    \mathbb{E}A_t\leq B_1-B_2\mathbb{E}H(\bar{X}^{\omega}_t).
\end{align*}
Taking the expectation in \eqref{eq:Ito_H}, remarking that $M_t$ is a martingale, and using Gronwall's lemma yields the desired result.
\end{proof}

%
%
%
%

\section{Mean-field limit}

Let $\phi_s,\phi_r:\mathbb{R}^+\rightarrow\mathbb{R}$ be two Lipschitz continuous functions satisfying, for some parameter $\delta>0$, the following conditions
\begin{equation}\label{def: phis phir}
    \forall x\in\mathbb{R}^+,\ \ \ \phi_s(x)^2+\phi_r(x)^2=1,\ \ \ \phi_r(x)=\left\{\begin{array}{ll}1&\text{ if }x\geq\delta,\\0&\text{ if }x\leq\delta/2 \end{array}\right.
\end{equation}
These functions describe the regions of space in which we either use a synchronous coupling ($\phi_s\equiv1$ and $\phi_r\equiv0$) and a reflection coupling ($\phi_s\equiv0$ and $\phi_r\equiv1$). Ideally, we would like to use 
$\phi_r(x)=\mathds{1}_{x>0}$, but the indicator function is not continuous, hence the reason we use a Lipschitz approximation.

Consider the initial conditions $(X_0^i,\omega_i)_{i\in\{1,...,N\}}$ for \eqref{eq:IPS}, and consider $N$ independent random variables $(\bar{X}_0^i)_{i=1,...,N}$ identically distributed according to $\bar{\rho}^1_0$ (recall Assumption~\ref{hyp:rho_0}). We know (see for instance Proposition~2.1 of \cite{COTFNT}) that there exists a least one permutation $\tau:\{1,...,N\}\rightarrow\{1,...,N\}$ such that
\begin{equation}\label{eq:wass_init}
    \mathcal{W}_1\left(\frac{1}{N}\sum_{i=1}^N \delta_{X^i_0},\frac{1}{N}\sum_{i=1}^N \delta_{\bar X^i_0}\right)=\frac{1}{N}\sum_{i=1}^N\left|X_0^i-\bar{X}_0^{\tau(i)}\right|.
\end{equation}
If there exists more that one such permutation we choose one of them uniformly. Up to renumbering, we assume $\tau(i)=i$ for all $i\in\{1,...,N\}$. The random variables $(\bar X^i_0,\omega_i)$ are then IID with distribution $\bar \rho_0$. Using these initial conditions, we now consider the following coupling
\begin{equation}\label{eq:coupling_s}
\left\{
\begin{array}{rcl}
    dX^i_t&=&F\left(X^i_t,\omega_i\right)dt+\frac{\alpha_N}{N}\sum_{j=1}^N\xi^{(N)}_{i,j}\Gamma\left(X^i_t,\omega_i,X^j_t,\omega_j\right)dt+\sqrt{2}\sigma\phi_s(|X^i_t-\bar{X}^i_t|)d\tilde{B}^i_t\\
    &&+\sqrt{2}\sigma\phi_r(|X^i_t-\bar{X}^i_t|)dB^i_t,\\
    d\bar{X}^i_t&=&F\left(\bar{X}^i_t,\omega_i\right)dt+p\int\Gamma(\bar{X}^i_t,\omega_i,x,\omega)\bar{\rho}_t(dx,d\omega)dt+\sqrt{2}\sigma \phi_s(|X^i_t-\bar{X}^i_t|)d\tilde{B}^i_t\\
    &&+\sqrt{2}\sigma(Id-2e^i_t(e^i_t)^T)\phi_r(|X^i_t-\bar{X}^i_t|)dB^i_t,
\end{array}
\right.
\end{equation}
where 
\begin{equation*}
    e^i_t:=\left\{\begin{array}{ll}\frac{X^i_t-\bar{X}^i_t}{|X^i_t-\bar{X}^i_t|}&\text{ if }X^i_t-\bar{X}^i_t\neq0,\\0&\text{ if }X^i_t-\bar{X}^i_t=0.\end{array}\right.,
\end{equation*}
and $\left(B^i_\cdot\right)_{i=1,...,N}$ and $\left(\tilde{B}^i_\cdot\right)_{i=1,...,N}$ are sequences of independent Brownian motions, and $\bar \rho_t$ is the distribution of the non linear diffusion \eqref{eq:NL}. In particular, Levy's characterization of Brownian motion ensures that $\left(\bar{X}^i_\cdot,\omega_i\right)_i$ are $N$ independent copies of the same diffusion process and thus $\bar{\rho}_t=\text{Law}(\bar{X}^1_t,\omega_1)=...=\text{Law}(\bar{X}^N_t,\omega_N)$.

Let us denote $Z^i_t=X^i_t-\bar{X}^i_t$. The following lemma concerning the dynamics of $|Z^i_t|$, which can be found in \cite{DEGZ20}, relies on dominated convergence and the fact that $\phi_r(x)$ is zero around $x=0$.


\begin{lemma}[Lemma~7 of \cite{DEGZ20}]\label{lem:dyn_norme_1}
For all $t\geq0$ and all $i\in\{1,...,N\}$,
\begin{equation*}
    d\left|Z^i_t\right|=\left(F\left(X^i_t,\omega_i\right)-F\left(\bar{X}^i_t,\omega_i\right)\right)\cdot e^i_t dt+A^i_tdt+2\sqrt{2}\sigma\phi_r(\left|Z^i_t\right|)e^i_t\cdot dB^i_t,
\end{equation*}
where $(A^i_t)_t$ is an adapted stochastic process such that
\begin{equation*}
A^i_t\leq\left|\frac{\alpha_N}{N}\sum_{j=1}^N\xi^{(N)}_{i,j}\Gamma\left(X^i_t,\omega_i,X^j_t,\omega_j\right)-p\int\Gamma(\bar{X}^i_t,\omega_i,x,\omega)\bar{\rho}_t(dx,d\omega)\right|.
\end{equation*}
\end{lemma}

Applying Itô-Tanaka's formula, as the function $f$ is $\mathcal{C}^1$ and piecewise $\mathcal{C}^2$ and concave, and relying on Lemma~\ref{lem:dyn_norme_1} we obtain
\begin{equation}\label{eq:dyn_f}
    df\left(\left|Z^i_t\right|\right)=f'\left(\left|Z^i_t\right|\right)\left(\left(F\left(X^i_t,\omega_i\right)-F\left(\bar{X}^i_t,\omega_i\right)\right)\cdot e^i_t+A^i_t\right)dt+4f''\left(\left|Z^i_t\right|\right)\sigma^2\phi_r^2(\left|Z^i_t\right|)dt+dM^i_t,
\end{equation}
where, with a slight abuse of notation, $f'$ denotes the left derivative of $f$ and $f''$ its almost everywhere defined second derivative, and $(M^i_t)_t$ is a continuous martingale (recall $f'$ is bounded). Let us define $\omega:\mathbb{R}^+\mapsto\mathbb{R}^+$ by
\begin{align*}
    \omega(r):=\sup_{s\in[0,r]}s\kappa_-(s).
\end{align*}
Relying on Assumption~\ref{hyp:F}, \eqref{ineq derivatives f} and \eqref{def: phis phir} we then get the following inequality:
\begin{align*}
\left(F\left(X^i_t,\omega_i\right)-F\left(\bar{X}^i_t,\omega_i\right)\right)&\cdot e^i_tf'\left(\left|Z^i_t\right|\right)+4f''\left(\left|Z^i_t\right|\right)\sigma^2\phi_r^2(\left|Z^i_t\right|)\\
\leq&-\left|Z^i_t\right|\kappa\left(\left|Z^i_t\right|\right)f'\left(\left|Z^i_t\right|\right)+4f''\left(\left|Z^i_t\right|\right)\sigma^2\phi_r^2(\left|Z^i_t\right|)\\
\leq&-2c\sigma^2 f\left(\left|Z^i_t\right|\right)\phi_r^2(\left|Z^i_t\right|)-\left|Z^i_t\right|\kappa\left(\left|Z^i_t\right|\right)f'\left(\left|Z^i_t\right|\right)\phi_s^2(\left|Z^i_t\right|)\\
\leq&-2c\sigma^2 f\left(\left|Z^i_t\right|\right)\phi_r^2(\left|Z^i_t\right|)+\omega(\delta)\\
\leq&-2c\sigma^2 f\left(\left|Z^i_t\right|\right)+\omega(\delta)+2c\sigma^2 f(\delta).
\end{align*}
We deduce that there exists an adapted process $K^i$ satisfying
\begin{align*}
    K^i_t\leq \omega(\delta)+2\sigma^2 c f(\delta),
\end{align*}
and such that for all $\tilde{\kappa}\in[0,2\sigma^2c]$
\begin{align}
d(e^{(2c\sigma^2 -\tilde{\kappa}) t}f\left(\left|Z^i_t\right|\right)=&e^{(2c\sigma^2 -\tilde{\kappa}) t}df\left(\left|Z^i_t\right|\right)+(2c\sigma^2 -\tilde{\kappa}e^{(2c\sigma^2 -\tilde{\kappa})t}f\left(\left|Z^i_t\right|\right)dt \nonumber\\
=&e^{(2c\sigma^2 -\tilde{\kappa}) t}(-\tilde{\kappa} f\left(\left|Z^i_t\right|\right)+ K^i_t+A^i_t)dt+e^{(2c\sigma^2 -\tilde{\kappa}) t}dM^i_t,\label{dyn expo with kappa}
\end{align}
with $A^i_t$ given in Lemma~\ref{lem:dyn_norme_1}.
The next step is to deal with $A^i_t$. We have
\begin{align*}
    &\left|\frac{\alpha_N}{N}\sum_{j=1}^N\xi^{(N)}_{i,j}\Gamma\left(X^i_t,\omega_i,X^j_t,\omega_j\right)-p\int\Gamma(\bar{X}^i_t,\omega_i,x,\omega)\bar{\rho}_t(dx,d\omega)\right|\\
    &\hspace{1cm}\leq\left|\frac{\alpha_N}{N}\sum_{j=1}^N\xi^{(N)}_{i,j}\left(\Gamma\left(X^i_t,\omega_i,X^j_t,\omega_j\right)-\Gamma\left(\bar{X}^i_t,\omega_i,\bar{X}^j_t,\omega_j\right)\right)\right|\\
    &\hspace{1.5cm}+\left|\frac{\alpha_N}{N}\sum_{j=1}^N\xi^{(N)}_{i,j}\left(\Gamma\left(\bar{X}^i_t,\omega_i,\bar{X}^j_t,\omega_j\right)-\int\Gamma(\bar{X}^i_t,\omega_i,x,\omega)\bar{\rho}_t(dx,d\omega)\right)\right|\\
    &\hspace{1.5cm}+\left|\left(\frac{\alpha_N}{N}\sum_{j=1}^N\xi^{(N)}_{i,j}-p\right)\int\Gamma(\bar{X}^i_t,\omega_i,x,\omega)\bar{\rho}_t(dx,d\omega)\right|\\
    &\hspace{1cm}=I_{1,i}+I_{2,i}+I_{3,i}.
\end{align*}
We deal with each of these three terms individually.


\paragraph{Dealing with $I_{1,i}$ : Lipschitz continuity of $\Gamma$.}
Using Assumption~\ref{hyp:gamma_lip},
\begin{align*}
    I_{1,i}=&\left|\frac{\alpha_N}{N}\sum_{j=1}^N\xi^{(N)}_{i,j}\left(\Gamma\left(X^i_t,\omega_i,X^j_t,\omega_j\right)-\Gamma\left(\bar{X}^i_t,\omega_i,\bar{X}^j_t,\omega_j\right)\right)\right|\\
    \leq&\frac{L_\Gamma \alpha_N}{N}\sum_{j=1}^N\xi^{(N)}_{i,j}\left(f(|X^i_t-\bar{X}^i_t|)+f(|X^j_t-\bar{X}^j_t|)\right)\\
    =&L_\Gamma  f(|Z^i_t|)\frac{\alpha_N}{N}\sum_{j=1}^N\xi^{(N)}_{i,j}+\frac{L_\Gamma \alpha_N}{N}\sum_{j=1}^N\xi^{(N)}_{i,j} f(|Z^j_t|).
\end{align*}
We then deduce, relying on Assumption~\ref{hyp:graph},
\begin{align*}
    \frac{1}{N}\sum_{i=1}^NI_{1,i}\leq \frac{L_\Gamma}{N} \sum_{i=1}^Nf(|Z^i_t|)\frac{\alpha_N}{N}d^{(N)}_i+\frac{L_\Gamma}{N}\sum_{j=1}^N  f(|Z^j_t|)\frac{ \alpha_N}{N}\tilde d^{(N)}_j
    \leq \frac{L_\Gamma D_{N,g}}{N}\sum_{i=1}^Nf(|Z^i_t|).
\end{align*}


\paragraph{Dealing with $I_{2,i}$ : some law of large numbers.}
Let us denote
\begin{align*}
    \bar{\Gamma}(x,\omega,y,\omega')=\Gamma\left(x,\omega,y,\omega'\right)-\int\Gamma(x,\omega,z,\tilde{\omega})\bar{\rho}_t(dz,d\tilde{\omega}).
\end{align*}
After expansion, we obtain (recall that we have made the hypothesis $\xi^{(N)}_{i,i}=0$)
\begin{align*}
    I_{2,i}^2 =&\Bigg|\frac{\alpha_N}{N}\sum_{j=1}^N\xi^{(N)}_{i,j}\bar{\Gamma}\left(\bar{X}^i_t,\omega_i,\bar{X}^j_t,\omega_j\right)\Bigg|^2\\
    =&\frac{\alpha_N^2}{N^2}\sum_{j=1,j\neq i}^N\xi^{(N)}_{i,j}\bar{\Gamma}\left(\bar{X}^i_t,\omega_i,\bar{X}^j_t,\omega_j\right)^2\\
    &+\frac{\alpha_N^2}{N^2}\sum_{j,k=1,j,k\neq i,j\neq k}^N\xi^{(N)}_{i,j}\xi^{(N)}_{i,k}\bar{\Gamma}\left(\bar{X}^i_t,\omega_i,\bar{X}^j_t,\omega_j\right)\bar{\Gamma}\left(\bar{X}^i_t,\omega_i,\bar{X}^k_t,\omega_k\right).
\end{align*}
The expectation of the last term conditioned to $(\bar X^i_t,\omega^i)$ is equal to $0$, and thus, relying in particular on Assumption~\ref{hyp:gamma_lip}-2 and Lemma~\ref{lem:semi-metric},
\begin{align*}
    \mathbb{E}\Big(I_{2,i}\Big|&\bar{X}^i_t,\omega_i\Big)\\
    \leq & \mathbb{E}\left(I_{2,i}^2\Big|\bar{X}^i_t,\omega_i\right)^{1/2}\\
    =&\mathbb{E}\left(\frac{\alpha_N^2}{N^2}\sum_{j=1,j\neq i}^N\xi^{(N)}_{i,j}\bar\Gamma\left(\bar{X}^i_t,\omega_i,\bar{X}^j_t,\omega_j\right)^2\Big|\bar{X}^i_t,\omega_i\right)^{1/2}\\
    \leq&\mathbb{E}\left(\frac{3\alpha_N^2}{N^2}\sum_{j=1,j\neq i}^N\xi^{(N)}_{i,j}\left|\Gamma\left(\bar{X}^i_t,\omega_i,\bar{X}^j_t,\omega_j\right)-\Gamma(0,\omega_i,0,\omega_j)\right|^2\Big|\bar{X}^i_t,\omega_i\right)^{1/2}\\
    &+\mathbb{E}\left(\frac{3\alpha_N^2}{N^2}\sum_{j=1,j\neq i}^N\xi^{(N)}_{i,j}\left|\int\left(\Gamma\left(0,\omega_i,0,\omega_j\right)-\Gamma(0,\omega_i,0,\omega)\right)\bar{\rho}_t(dx,d\omega)\right|^2\Big|\bar{X}^i_t,\omega_i\right)^{1/2}\\
    &+\mathbb{E}\left(\frac{3\alpha_N^2}{N^2}\sum_{j=1,j\neq i}^N\xi^{(N)}_{i,j}\left|\int\left(\Gamma(0,\omega_i,0,\omega)-\Gamma(\bar{X}^i_t,\omega_i,x,\omega)\right)\bar{\rho}_t(dx,d\omega)\right|^2\Big|\bar{X}^i_t,\omega_i\right)^{1/2}\\
    \leq& 2\left[\frac{6 L_\Gamma^2C_f^2\alpha_N^2}{N^2} d^{(N)}_i \left(|\bar{X}^i_t|^2+\int |x|^2 \bar{\rho}_t(dx,d\omega)\right)\right]^{1/2}+\left[\frac{12L_\Gamma^2\alpha_N^2}{N^2}d^{(N)}_i\int |\omega|^2 \bar{\rho}_t(dx,d\omega)\right]^{1/2}.
\end{align*}
We deduce, recalling Lemma~\ref{lem:mom_nl},
\begin{align*}
    \mathbb{E}\left(I_{2,i}\right)\leq 2\sqrt{3}L_\Gamma(2C_f\bar{\mathcal{ C}}_2^{1/2}+C_{dis}^{1/2})\sqrt{\frac{\alpha_N D_{N,g}}{N}}.
\end{align*}
Remark that if $\Gamma$ satisfies Assumption~\ref{hyp:gamma_lip}-2-bis, then we simply have
\begin{align*}
    \mathbb{E}&\left(I_{2,i}\right)
    \leq 2L_\infty \sqrt{\frac{\alpha_N D_{N,g}}{N}}.
\end{align*}


\paragraph{Dealing with $I_{3,i}$ : convergence of the graph.}
We immediately get
\begin{align*}
    \mathbb{E}\left(I_{3,i}\right)\leq& I_{N,g}\mathbb{E}\left(\left|\int\Gamma(\bar{X}^i_t,\omega_i,x,\omega)\bar{\rho}_t(dx,d\omega)\right|\right),
\end{align*} 
and thus, if Assumption~\ref{hyp:gamma_lip}-2-bis holds, this directly implies
$\mathbb{E}\left(I_{3,i}\right)\leq L_\infty I_{N,g}$.
Otherwise, if Assumption~\ref{hyp:gamma_lip}-2 holds, we obtain
\begin{align*} 
    \mathbb{E}\left(\left|\int\Gamma(\bar{X}^i_t,\omega_i,x,\omega)\bar{\rho}_t(dx,d\omega)\right|\right)\leq& \mathbb{E}\left(\int\left|\Gamma(\bar{X}^i_t,\omega_i,x,\omega)-\Gamma(0,\omega_i,0,\omega)\right|\bar{\rho}_t(dx,d\omega)\right)\\
    &+\mathbb{E}\left(\int\left|\Gamma(0,\omega_i,0,\omega)-\Gamma(0,0,0,0)\right|\bar{\rho}_t(dx,d\omega)\right)\\
    \leq&2L_\Gamma \left(C_f\mathbb{E}\left|\bar{X}^i_t\right|+\mathbb{E}\left|\omega_i\right|\right).
\end{align*}
So, using Lemma~\ref{lem:mom_nl}, we get
\begin{align*}
    \mathbb{E}\left(I_{3,i}\right)\leq 2L_\Gamma \left(C_f\bar{\mathcal{C}}_2^{1/2}+\mathcal{C}_{2,\omega}^{1/2}\right)I_{N,g}.
\end{align*}

\paragraph{Conclusion}
Recalling \eqref{dyn expo with kappa} and choosing $\tilde{\kappa}=L_\Gamma D_{N,g}$ we obtain
\begin{align*}
    d\left(e^{(2\sigma^2 c-L_\Gamma D_{N,g})t}f(|Z^{i}_t|)\right)=e^{(2\sigma^2 c-L_\Gamma D_{N,g})t}\tilde{K}^i_tdt+e^{(2\sigma^2 c-L_\Gamma D_{N,g})t}dM^i_t,
\end{align*}
where there exists a constant $C_0$, depending on the parameters as well as possibly on $\bar{\rho}_0$, but that do not depend on $N$ and on the graph, such that
\begin{align*}
    \frac{1}{N}\sum_{i=1}^N\mathbb{E}\tilde{K}^i_t\leq C_0 L_\Gamma\left(\sqrt{\frac{\alpha_N D_{N,g}}{N}}+I_{N,g}\right)+\omega(\delta)+2\sigma^2 c f(\delta).
\end{align*}
Then
\begin{align*}
\mathbb{E}\left(\frac{e^{(2\sigma^2 c-L_\Gamma D_{N,g})t}}{N}\sum_{i=1}^Nf\left(\left|Z^i_t\right|\right)\right)&-\mathbb{E}\left(\frac{1}{N}\sum_{i=1}^Nf\left(\left|Z^i_0\right|\right)\right)\\
\leq& \frac{e^{(2\sigma^2 c-L_\Gamma D_{N,g})t}-1}{2\sigma^2 c-L_\Gamma D_{N,g}}\left(C_0 L_\Gamma\left(\sqrt{\frac{\alpha_N D_{N,g}}{N}}+I_{N,g}\right)+\omega(\delta)+2\sigma^2 c f(\delta)\right),
\end{align*}
i.e.
\begin{align*}
\mathbb{E}\left(\frac{1}{N}\sum_{i=1}^N\left|X^i_t-\bar{X}^i_t\right|\right)\leq& \frac{C_fe^{-(2\sigma^2 c-L_\Gamma D_{N,g})t}}{c_f}\mathbb{E}\left(\frac{1}{N}\sum_{i=1}^N\left|X^i_0-\bar{X}^i_0\right|\right)\\
&+\frac{1}{c_f(2\sigma^2 c-L_\Gamma D_{N,g})}\left(C_0 L_\Gamma\left(\sqrt{\frac{\alpha_N D_{N,g}}{N}}+I_{N,g}\right)+\omega(\delta)+2\sigma^2 c f(\delta)\right)
\end{align*}
Thus, denoting $\bar{\mu}^N_t$ the empirical measure associated with the system of independent non-linear particles $\left((\bar{X}^1_t,\omega_1),...,(\bar{X}^N_t,\omega_N)\right)$, we obtain, for $c_\Gamma=\sigma^2 c$ and $L_\Gamma \leq c_\Gamma/D_{N,g}$,
\begin{align*}
\mathbb{E}\mathcal{W}_1\left(\mu^N_t,\bar{\mu}^N_t\right)\leq& \frac{C_fe^{-\sigma^2 c 
t}}{c_f}\mathbb{E}\left(\frac{1}{N}\sum_{i=1}^N\left|X^i_0-\bar{X}^i_0\right|\right)\\
&+\frac{1}{c_f  \sigma^2 c}\left(C_0 L_\Gamma\left(\sqrt{\frac{\alpha_N D_{N,g}}{N}}+I_{N,g}\right)+\omega(\delta)+2\sigma^2 c f(\delta)\right).
\end{align*}
Notice that 
\begin{align*}
    \mathcal{W}_1\left(\mu^N_0,\bar{\mu}^N_0\right)=&\mathcal{W}_1\left(\frac{1}{N}\sum_{i=1}^N\delta_{(X^i_0,\omega_i)},\frac{1}{N}\sum_{i=1}^N\delta_{(\bar{X}^i_0,\omega_i)}\right)\\
    =&\min_{\tau\text{ permutation}}\left\{\frac{1}{N}\sum_{i=1}^N|X^i_0-\bar{X}^{\tau(i)}_0|+|\omega_i-\omega_{\tau(i)}|\right\}\\
    =&\frac{1}{N}\sum_{i=1}^N|X^i_0-\bar{X}^i_0|,
\end{align*}
as both terms to minimize are minimal for $\tau$ the identity. By having $\delta\rightarrow0$,  we thus have
\begin{align*}
\mathbb{E}\mathcal{W}_1\left(\mu^N_t,\bar{\mu}^N_t\right)\leq& \frac{C_fe^{-\sigma^2 c t}}{c_f}\mathbb{E}\mathcal{W}_1\left(\mu^N_0,\bar{\mu}^N_0\right)\\
&+\frac{1}{c_f \sigma^2 c }\left(C_0L_\Gamma\left(\sqrt{\frac{\alpha_N}{N}}+I_{N,g}\right)\right).
\end{align*}
Since $\left((\bar{X}^1_t,\omega_1),...,(\bar{X}^N_t,\omega_N)\right)$ are N independent random variables with law $\bar{\rho}_t$ by construction, and since $\bar{\rho}_t$ admits a second moment, Theorem~1 of \cite{FG15} yields the existence of a constant $C$, depending only on the dimensions $d$ and $d'$, such that 
\begin{align*}
    \mathbb{E}\mathcal{W}_1\left(\bar{\mu}^N_t,\bar{\rho}_t\right)\leq& C\left(C_{\text{dis}}+\bar{\mathcal{C}}_2\right)^{1/2}\left\{
    \begin{array}{ll}
    N^{-\frac{1}{2}}+N^{-\frac{1}{3}}&\text{ if }d+d'=1,\\
    N^{-\frac{1}{2}}\log(1+N)+N^{-\frac{1}{3}}&\text{ if }d+d'=2,\\
    N^{-\frac{1}{2}}+N^{-\frac{1}{d+d'}}&\text{ if }d+d'\geq3.
    \end{array}
    \right. 
\end{align*}
The convergence rates could be improved (with respective rates $N^{-\frac12}$, $N^{-\frac12 }\log (1+N)$ and $N^{-\frac{1}{d+d'}}$) provided we can prove uniform in time bounds on a moment of order $q>2$ for $\bar{\rho}_t$. This can be done, but requires a similar great moment assumption on the initial distribution $\bar{\rho}_0$.

\appendix

\section{Graph estimates}\label{sec:app}

\begin{lemma}
Let us fix an integer $r$, consider an integer $m$, define the total size of the population $N=mr$, and define independent random variables $\xi^{(N,k,k')}_{i,j}$ for $k,k'\in \{1,\ldots, r\}$ and $i,j\in \{1,\ldots,m\}$ such that $\xi^{(N,k,k)}_{i,j}$ are of Bernoulli distribution with parameter $q^{k,k}_N=q_N$ satisfying $\frac{1}{q_N} = o\left(\frac{N}{\log N}\right)$, while for $k\neq k'$  $\xi^{(N,k,k')}_{i,j}$ are of Bernoulli distribution with parameter $q^{k,k'}_N$ satisfying $q^{k,k'}_N=o(q_N)$. Then, defining $d_i^{(N,k)} =\sum_{k'=1}^r \sum_{j=1}^m  \xi^{(N,k,k')}_{i,j}$ and $\tilde d_i^{(N,k)} =\sum_{k'=1}^r \sum_{j=1}^m  \xi^{(N,k,k')}_{j,i}$, there exists a constant $C$ such that
\begin{align*}
\limsup_{m\rightarrow \infty} \sup_{k\in \{1,\ldots,r\}} \sup_{i\in \{1,\ldots m\}} \frac{1}{q_N}\left(\frac{d^{(N,k)}_i}{N}+ \frac{\tilde d^{(N,k)}_i}{N}\right)\leq C,
\end{align*}
and moreover
\begin{align*}
\sup_{k\in \{1,\ldots,r\}} \sup_{i\in\{1,...,m\}}\left|\frac{d^{(N,k)}_i}{N q_N }-\frac1r\right|\xrightarrow[N\rightarrow\infty]{a.s}0.
\end{align*}
\end{lemma}

\begin{proof}
We only prove the second estimate, the first one being a consequence of the second claim and the fact that $d^{(N,k)}_i$ and $\tilde d^{(N,k)}_i$ have the same distribution. Remarking that the independent random variables $Z^{(N,k,k')}_{i,j}:=\frac{1}{Nq_N}\left(\xi^{(N,k,k')}_{i,j}-q^{k,k'}_N\right)$ satisfy $\left|Z^{(N,k,k')}_{i,j}\right|\leq \frac{1}{Nq_N}$ and $\mathbb{E}\left[\left|Z^{(N,k,k')}_{i,j}\right|^2\right]\leq \frac{q_N^{k,k'}}{N^2 q_N^2}$, Benrstein inequality leads to
\begin{align*}
    \mathbb{P} \left(\left|\sum_{k'=1}^r \sum_{j=1}^m Z^{(N,k,k')}_{i,j}\right| > t\right)\leq 2 \exp\left(-\frac12 \frac{N q_N t^2}{\sum_{k'=1}^r \frac{m q_N^{k,k'}}{N q_N} +\frac{t}{3}}\right).
\end{align*}
Taking $t=\sqrt{\frac{c\log N}{Nq_N}}$ for some positive constant $c$ and remarking that $q_N^{k,k'}\leq q_N$ and $\frac13 \sqrt{\frac{c\log N}{Nq_N}}\leq 1$ for $N$ large enough and we get
\begin{align*}
    \mathbb{P} \left(\left|\sum_{k'=1}^r \sum_{j=1}^m Z^{(N,k,k')}_{i,j}\right| > \sqrt{\frac{c\log N}{Nq_N}}\right)\leq 2 \exp\left(-\frac12 \frac{c\log N}{\sum_{k'=1}^r \frac{m q_N^{k,k'}}{N q_N} +\frac{1}{3}\sqrt{\frac{c\log N}{Nq_N}}}\right)\leq 2 N^{-\frac{c}{4}}.
\end{align*}
So
\begin{align*}
   \mathbb{P} \left(\sup_{k\in \{1,\ldots,r\}} \sup_{i\in\{1,...,m\}}\left|\frac{d^{(N,k)}_i}{N q_N }-\sum_{k'=1}^r \frac{mq^{k,k'}_N}{N q_N}\right| > \sqrt{\frac{c\log N}{Nq_N}}\right) \leq 2 N^{1-\frac{c}{4}},
\end{align*}
and we conclude by applying Borel-Cantelli Lemma, taking $c$ large enough and noting that $\sum_{k'=1}^r \frac{mq^{k,k'}_N}{N q_N}$ converges to $\frac{1}{r}$ as $N$ goes to infinity (recall that $q_N^{k,k}=q_N$).
\end{proof}

\subsection*{Acknowledgement}

P.LB. and C.P. acknowledge the support of ANR-17-CE40-0030 (EFI). C.P. also acknowledges the support of ANR-19-CE40-0023 (PERISTOCH).

\subsection*{Availability of Data and Materials}
Data sharing not applicable to this article as no datasets were generated or
analysed during the current study.

\subsection*{Conflicts of interest}
The authors declare no conflict of interest.

\nocite{*}
\bibliographystyle{amsplain}
\bibliography{biblio_PoC_graph}

\end{document}